\documentclass[10pt]{amsart}
\usepackage{amsmath}
\usepackage{amsthm}
\usepackage{amsfonts}
\usepackage{mathrsfs}
\usepackage{amssymb}
\usepackage{amscd}
\usepackage{pgf}
\usepackage{tikz}
\usetikzlibrary{cd}
\usepackage{amscd}

\expandafter\let\csname ver@amsthm.sty\endcsname\relax
\let\theoremstyle\relax

\usepackage[utf8]{inputenc}
\usepackage[
    breaklinks,
    colorlinks,
    citecolor=teal,
    linkcolor=teal,
    urlcolor=teal,
    pagebackref=true,
    hyperindex
    ]{hyperref}
\usepackage{fancyhdr}
\usepackage[
    hscale=0.7,
    vscale=0.75,
    headheight=13pt,
    centering,
    ]{geometry}
\usepackage{amsmath}
\usepackage{amsthm}
\usepackage{amssymb}
\usepackage{mathtools}
\usepackage{stmaryrd}
\usepackage{mathdots}
\usepackage[all,cmtip]{xy}
\usepackage{xcolor}
\usepackage{framed}
\usepackage{pgffor}
\usepackage{fnpct}
\usepackage[capitalize]{cleveref}

\theoremstyle{plain}

\newtheorem{theorem}{Theorem}[section]
\newtheorem{proposition}[theorem]{Proposition}
\newtheorem{lemma}[theorem]{Lemma}
\newtheorem{corollary}[theorem]{Corollary}

\newtheorem{conjecture}[theorem]{Conjecture}

\theoremstyle{definition}

 \newtheorem{Def}[theorem]{Definition}
    \newtheorem{Exa}[theorem]{Example}
    \newtheorem{Rem}[theorem]{Remark}

		\newtheorem{Question}[theorem]{Question}

\numberwithin{figure}{section}

\numberwithin{equation}{section}

\makeatletter

\def\@myMR[#1 #2]{\relax\ifhmode\unskip\spacefactor3000 \space\fi
  \MRhref{#1}{MR\,#1}}
\renewcommand\MR[1]{\@myMR[#1 ]}
\renewcommand{\MRhref}[2]{{\tiny%
  \href{http://www.ams.org/mathscinet-getitem?mr=#1}{#2}}%
}
\renewcommand*{\backref}[1]{}
\renewcommand*{\backrefalt}[4]{%
    \tiny%
    ({
    \ifcase #1 not cited%
          \or cit.\ on p.~#2%
          \else cit.\ on pp.~#2%
    \fi%
    })\\[-.6em]}

\def\maketitle{\par
  \@topnum\z@ 
  \@setcopyright
  \thispagestyle{empty}
  \ifx\@empty\shortauthors \let\shortauthors\shorttitle
  \else \andify\shortauthors
  \fi
  \@maketitle@hook
  \begingroup
  \@maketitle
  \toks@\@xp{\shortauthors}\@temptokena\@xp{\shorttitle}%
  \toks4{\def\\{ \ignorespaces}}
  \edef\@tempa{%
    \@nx\markboth{\the\toks4
      \@nx\MakeUppercase{\the\toks@}}{\the\@temptokena}}%
  \@tempa
  \endgroup
  \c@footnote\z@
    \renewcommand{\footnoterule}{%
      \kern -3pt
      \hrule width \textwidth height .5pt
      \kern 2pt
    }
  {
    \renewcommand\thefootnote{}
    \vspace{-2em}
    \footnote{
      \par\vspace{-1.2em}\noindent
      \def\@footnotetext##1{\noindent{\footnotesize##1}\par}%
      \let\@makefnmark\relax  \let\@thefnmark\relax
      \ifx\@empty\@date\else \@footnotetext{\@setdate}\fi
      \ifx\@empty\@subjclass\else \@footnotetext{\@setsubjclass}\fi
      \ifx\@empty\@keywords\else \@footnotetext{\@setkeywords}\fi
      \ifx\@empty\thankses\else \@footnotetext{%
        \def\par{\let\par\@par}\@setthanks}%
      \fi
    }
    \addtocounter{footnote}{-1}
  }
  \@cleartopmattertags
}

\def\@adminfootnotes{\@empty}

\def\@settitle{\begin{center}%
  \baselineskip14\p@\relax
    \bfseries
\Large
  \@title
  \end{center}%
}

\def\@setauthors{%
  \begingroup
  \def\thanks{\protect\thanks@warning}%
  \trivlist
  \centering\footnotesize \@topsep30\p@\relax
  \advance\@topsep by -\baselineskip
  \item\relax
  \author@andify\authors
  \def\\{\protect\linebreak}%
  \large{\authors}%
  \ifx\@empty\contribs
  \else
    ,\penalty-3 \space \@setcontribs
    \@closetoccontribs
  \fi
  \endtrivlist
  \endgroup
}

\def\@setaddresses{\par
  \nobreak \begingroup
\footnotesize
  \def\author##1{\end{minipage}\hskip 1sp \begin{minipage}{.5\textwidth}\raggedright%
    ~\\[2em]{\bf##1}\\[.5em]%
  }%
  \interlinepenalty\@M
  \def\address##1##2{\begingroup
    {\ignorespaces##2}\endgroup\\[.5em]}%
  \def\curraddr##1##2{\begingroup
    \@ifnotempty{##2}{\nobreak\indent\curraddrname
      \@ifnotempty{##1}{, \ignorespaces##1\unskip}\/:\space
      ##2\par}\endgroup}%
  \def\email##1##2{\begingroup
    \@ifnotempty{##2}{\nobreak\indent
      \@ifnotempty{##1}{, \ignorespaces##1\unskip}
      \ttfamily##2\par}\endgroup}%
  \def\urladdr##1##2{\begingroup
    \def~{\char`\~}%
    \@ifnotempty{##2}{\nobreak\indent\urladdrname
      \@ifnotempty{##1}{, \ignorespaces##1\unskip}\/:\space
      \ttfamily##2\par}\endgroup}%
  \setlength{\parindent}{0pt}%
  \vfill%
  {
  \begin{minipage}{0mm}
  \addresses
  \end{minipage}
  }
  \endgroup
}

\renewcommand{\author}[2][]{%
  \ifx\@empty\authors
    \gdef\authors{#2}%
    \g@addto@macro\addresses{\author{#2}}%
  \else
    \g@addto@macro\authors{\and#2}%
    \g@addto@macro\addresses{\author{#2}}%
  \fi
  \@ifnotempty{#1}{%
    \ifx\@empty\shortauthors
      \gdef\shortauthors{#1}%
    \else
      \g@addto@macro\shortauthors{\and#1}%
    \fi
  }%
}
\edef\author{\@nx\@dblarg
  \@xp\@nx\csname\string\author\endcsname}

\def\@secnumfont{\@empty}

\def\section{\@startsection{section}{1}%
  \z@{.7\linespacing\@plus\linespacing}{.5\linespacing}%
  {\large\bfseries\centering}}

\makeatother

\usepackage{enumerate}
\usepackage{hyperref}

\title{An update on Haiman's conjectures}
\author{Alex Abreu}

\address{
    Instituto de Matemática e Estatística\\
    Universidade Federal Fluminense\\
    Rua Prof. M. W. de Freitas, S/N\\
    24210-201 Niterói, Rio de Janeiro, Brasil
}

\email{alexbra1@gmail.com}

\author{Antonio Nigro}

\address{
    Instituto de Matemática e Estatística\\
    Universidade Federal Fluminense\\
    Rua Prof. M. W. de Freitas, S/N\\
    24210-201 Niterói, Rio de Janeiro, Brasil
}

\email{antonio.nigro@gmail.com}
\date{}

\usepackage{color}
\definecolor{forestgreen}{rgb}{0.13, 0.55, 0.13}

\newcommand{\col}{\colon}

\newcommand{\m}{\mathbf{m}}

\DeclareMathOperator{\ch}{ch}

\DeclareMathOperator{\Coess}{Coess}

\DeclareMathOperator{\csf}{csf}

\newcommand{\flag}{\mathcal{B}}

\newcommand{\h}{\mathcal{Y}}

\newcommand{\dw}{\dot{w}} 
\newcommand{\ds}{\dot{s}} 
\newcommand{\dz}{\dot{z}}

\newcommand{\dyckpath}[2]{
\draw[line width=2.5pt] (#1) foreach \dir in {#2}{ -- ++(\dir*90:1)};
}

\begin{document}
\maketitle
\begin{abstract}
    We revisit Haiman's conjecture on the relations between characters of Kazdhan-Lusztig basis elements of the Hecke algebra over $S_n$. The conjecture asserts that, for purposes of character evaluation, any Kazhdan-Lusztig basis element is reducible to a sum of the simplest possible ones (those associated to so-called codominant permutations). When the basis element is associated to a smooth permutation, we are able to give a geometric proof of this conjecture. On the other hand, if the permutation is singular, we provide a counterexample.
  
\end{abstract}

\section{Introduction}

    The group algebra $\mathbb{C}[S_n]$  admits a $q$-deformation called the \emph{Hecke algebra} $H_n$, constructed as follows. Since every $w\in S_n$ can be written as a product of simple transpositions $(i,i+1)$, the group algebra $\mathbb{C}[S_n]$ can be described as the $\mathbb{C}$-algebra generated by $\{T_{s}\}$, where $s$ runs through all simple transpositions, with the relations
      \begin{align*}
          T_s^2=&1&&\text{ for every simple transposition $s$,}\\
          T_sT_{s'}=&T_{s'}T_s&&\text{ for every $s=(i,i+1)$ and $s'=(j,j+1)$ such that $|i-j|>1$,}\\
          T_sT_{s'}T_s=&T_{s'}T_sT_{s'}&&\text{ for every  $s=(i,i+1)$ and $s'=(j,j+1)$ such that $|i-j|=1$.}
      \end{align*}
       The algebra $H_n$ has the same generators as $\mathbb{C}[S_n]$ but with slightly different relations, although we abuse the notation and still write $T_s$ for these generators. Namely, $H_n$ is the $\mathbb{C}(q^{\frac{1}{2}})$-algebra\footnote{Usually, the definition is over $\mathbb{Z}[q^{\frac{1}{2}},q^{-\frac{1}{2}}]$.} generated by $\{T_s\}$, with the relations
    \begin{align*}
          T_s^2=&(q-1)T_s+q&&\text{ for every simple transposition $s$,}\\
          T_sT_{s'}=&T_{s'}T_s&&\text{ for every $s=(i,i+1)$ and $s'=(j,j+1)$ such that $|i-j|>1$,}\\
          T_sT_{s'}T_s=&T_{s'}T_sT_{s'}&&\text{ for every  $s=(i,i+1)$ and $s'=(j,j+1)$ such that $|i-j|=1$.}
      \end{align*}
      When $q=1$, we recover the group algebra $\mathbb{C}[S_n]$. 
      Since each $w\in S_n$ has a (non-unique) reduced expression $w=s_1s_2\ldots s_{\ell(w)}$ in terms of simple transpositions, the product 
      \[
      T_w:=T_{s_1}T_{s_2}\ldots T_{s_{\ell(w)}},
      \]
      is well defined, independent of the choice of reduced expression for $w$. Then as a $\mathbb{C}(q^{\frac{1}{2}})$-vector space, $\{T_w\}_{w\in S_n}$ is a basis of $H_n$.\par

        To introduce the Kazhdan-Lusztig basis, we first define the \emph{Bruhat order} of $S_n$: The \emph{length} $\ell(w)$ of $w$ is the number of inversions of $w$, and given $z,w\in S_n$, we say that $z\leq w$ if for some (equivalently, for every) reduced expression $w=s_1\ldots s_{\ell(w)}$ there exist $1\leq i_1<i_2<\ldots< i_k\leq \ell(w)$ such that $z=s_{i_1}\ldots s_{i_k}$. Then letting $\iota$ denote the involution of $H_n$ given by
      \begin{align*}
      \iota\col H_n&\to H_n\\
      q^{\frac{1}{2}}&\mapsto q^{-\frac{1}{2}}\\
      T_w&\mapsto T_{w^{-1}}^{-1},
      \end{align*}
       the \textit{Kazhdan-Lusztig basis} $\{C'_w\}_{w\in S_n}$ of $H_n$ is defined by the following properties:
      \begin{equation}
          \label{eq:C'wdef}
            \begin{aligned}
      \iota(C'_w)&=C'_w,\\
      q^{\frac{\ell(w)}{2}}C'_w&=\sum_{z\leq w}P_{z,w}(q)T_z,    
      \end{aligned}
      \end{equation}
      where $P_{z,w}(q)\in \mathbb{Z}[q]$, $P_{w,w}(q)=1$ and $\deg(P_{z,w})<\frac{\ell(w)-\ell(z)}{2}$ for every $z\neq w$. The existence of such a basis is proved in \cite{KL} and the polynomials $P_{z,w}(q)$ are called \emph{Kazhdan-Lusztig polynomials}.\par

    The Kazdhan-Lusztig elements and polynomials are closely related to the geometry of Schubert varieties in the flag variety. The flag variety $\flag$ is the projective variety parametrizing flags of vector subspaces of $\mathbb{C}^n$, that is
       \[
       \flag=\{V_1\subset V_2\subset\ldots\subset V_n=\mathbb{C}^n; \dim_{\mathbb{C}}(V_i)=i\}.
       \]
       We often abbreviate and write $V_\bullet$ to denote $V_1\subset\ldots \subset V_n$.  For each permutation $w$, the relative Schubert variety $\Omega_w$ and its open cell $\Omega_w^\circ$ are defined as 
       \begin{equation}
       \label{eq:def_schubert}
       \begin{aligned}
       \Omega_w&:=\{(F_\bullet,V_\bullet); \dim V_i\cap F_j\geq r_{i,j}(w)\text{ for }i,j = 1,\ldots, n\}\subset \flag\times \flag,\\
       \Omega^\circ_w&:=\{(F_\bullet,V_\bullet); \dim V_i\cap F_j= r_{i,j}(w)\text{ for }i,j = 1,\ldots, n\}\subset \flag\times \flag,
       \end{aligned}
              \end{equation}
       where
       \[
       r_{i,j}(w).: = |\{k;k\leq i,w(k)\leq j\}|.
       \]
       Then $\Omega_w=\bigsqcup_{z\leq w} \Omega_z^\circ$, where the disjoint union is taken over all permutations smaller than $w$ in the Bruhat order of $S_n$. \par 
      
       The Kazdhan-Lusztig polynomial $P_{z,w}(q)$ measures the singularity of $\Omega_w$ at $\Omega_z^\circ$, in the sense that $P_{z,w}(q) = \sum_{i}\dim H^i((IC_{\Omega_w})_p)q^{\frac{i}{2}}$, where $IC_{\Omega_w}$ is the intersection homology complex of $\Omega_w$ and $p$ is a point in $\Omega_z^\circ$.
    
       Note that not all conditions in Equation \eqref{eq:def_schubert} defining $\Omega_w$ are necessary: The \emph{coessential set} $\Coess(w)$ of $w$ is the smallest set of pairs $(i,j)$ such that 
       \[
       \Omega_w = \{(F_\bullet,V_\bullet); \dim V_i\cap F_j\geq r_{i,j}(w)\}.
       \]
       Equivalently, we have
        \[
        \Coess(w):=\{(i,j); w(i)\leq j<w(i+1),\; w^{-1}(j)\leq i< w^{-1}(j+1)\}.
        \]
       If a permutation $w$ satisfies $r_{i,j}(w)=\min(i,j)$ for every $(i,j)\in \Coess(w)$, we say that $\Omega_w$ is \emph{defined by inclusions}. Indeed, the condition $\dim V_i\cap F_j = r_{i,j}(w)$ is equivalent to either $V_i\subset F_j$ or $F_j\subset V_i$. If $\Omega_w$ is defined by inclusions and for every $(i_0,j_0), (i_1,j_1)\in \Coess(w)$ with $i_0\leq  j_0$ and $j_1\leq  i_1$ we have that either $j_0\leq  j_1$ or $i_1 \leq   i_0$, then we say that $\Omega_w$ is \emph{defined by non-crossing inclusions}.
    
    Given $w\in S_n$, it is well-known that the following conditions are equivalent:
    \begin{enumerate}
        \item $P_{e,w}(q)=1$,
        \item $\Omega_w$ is smooth,
        \item $\Omega_w$ is defined by non-crossing inclusions,
        \item $w$ avoids the patterns $3412$ and $4231$.
    \end{enumerate}
    \begin{Def}
    \label{def:smoothperm}
    A permutation satisfying any of the conditions above is called \emph{smooth}, otherwise it is called \emph{singular}.
    \end{Def}
    
    If the inclusions defining $\Omega_w$ are all of the form $V_i\subset F_j$, that is, if $i \leq j$ for every $(i,j)\in \Coess(w)$, we say that $w$ is \emph{codominant}. Codominant permutations are precisely the $312$-avoiding permutations, and there is a natural bijection between codominant permutations and Hessenberg functions (or Dyck paths), that is,  non-decreasing functions $\m\col [n]\to [n]$ satisfying $\m(i)\geq i$ for $i=1,\ldots, n$. The codominant permutation $w_\m$ associated to $\m$ is the lexicographically greatest permutation satisfying $w_\m(i)\leq \m(i)$ for all $i\in [n]$ (see Figure \ref{fig:cod_permutation}).
     
    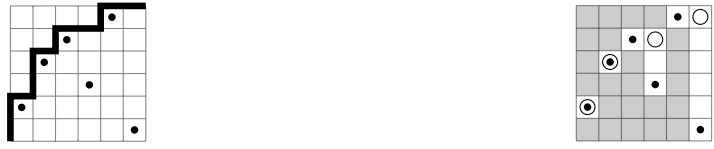
\begin{figure}[h]

\begin{minipage}{0.49\linewidth}
\centering
\begin{tikzpicture}
\begin{scope}[scale=0.3]
\draw[help lines] (0,0) grid +(6,6);
\dyckpath{0,0}{1,1,0,1,1,0,1,0,0,1,0,0}
\node at  (0.5,1.5) [shape=circle, fill=black, inner sep=1pt] {};
\node at  (1.5,3.5) [shape=circle, fill=black, inner sep=1pt] {};
\node at  (2.5,4.5) [shape=circle, fill=black, inner sep=1pt] {};
\node at  (3.5,2.5) [shape=circle, fill=black, inner sep=1pt] {};
\node at  (4.5,5.5) [shape=circle, fill=black, inner sep=1pt] {};
\node at  (5.5,0.5) [shape=circle, fill=black, inner sep=1pt] {};
\end{scope}
\end{tikzpicture}
\end{minipage}        
\begin{minipage}{0.49\linewidth}
\centering
\begin{tikzpicture}
\begin{scope}[scale=0.3]
\draw[help lines] (0,0) grid +(6,6);
\draw[fill=black, fill opacity=0.2, draw opacity=0] (0,0) -- (5,0) -- (5,5) -- (4,5) -- (4,2) -- (1,2) -- (1, 1) -- (0,1) -- (0,0);

\draw[fill=black, fill opacity=0.2, draw opacity=0] (0,2) -- (3,2) -- (3,4) -- (2,4) -- (2,3) -- (1,3) -- (1, 4) -- (2,4) -- (2,5) -- (4,5) -- (4,6) -- (0,6) -- (0,2);

\node at  (0.5,1.5) [shape=circle, fill=black, inner sep=1pt] {};
\node at  (0.5,1.5) [shape=circle, draw,  inner sep=2pt] {};
\node at  (1.5,3.5) [shape=circle, fill=black, inner sep=1pt] {};
\node at  (1.5,3.5) [shape=circle, draw,  inner sep=2pt] {};
\node at  (2.5,4.5) [shape=circle, fill=black, inner sep=1pt] {};
\node at  (3.5,2.5) [shape=circle, fill=black, inner sep=1pt] {};
\node at  (4.5,5.5) [shape=circle, fill=black, inner sep=1pt] {};
\node at  (3.5,4.5) [shape=circle, draw,  inner sep=2pt] {};
\node at  (5.5,0.5) [shape=circle, fill=black, inner sep=1pt] {};
\node at  (5.5,5.5) [shape=circle, draw,  inner sep=2pt] {};
\end{scope}
\end{tikzpicture}
\end{minipage}        
    \caption{The graphical representation of the Dyck path associated to the Hessenberg function $\m=(2,4,5,5,6,6)$ and of the codominant permutation $w_\m = 245361$. To find the coessential set of $w$, we remove every square that is below or to the left of a dot (greyed out in the picture). The coessential set is then the set of squares that are in the upper-right corner of the connected components of the remaining figure, the squares marked with a circle, $\Coess(w) = \{ ( 1,2), (2,4), (4, 5), (6,6)\}$.}
    \label{fig:cod_permutation}
    \end{figure}

    For codominant permutations $w_\m$, the Schubert varieties are characterized by
    \[
    \Omega_{w_\m} = \{(V_\bullet, F_\bullet); V_i\subset F_{\m(i)}\}.
    \]
    
    The bijection between codominant permutations and Hessenberg functions can be extended to map from the set of smooth permutations to the set of Hessenberg functions. Indeed, for every smooth permutation $w$, we can define a Hessenberg function $\m_w$ as follows. Let $I\subset [n]$ be the subset of indices $i$ such that there exists $j\geq i$ with either $(i,j)\in \Coess(w)$ or $(j,i)\in \Coess(w)$. We define $\m_w$ by the conditions $\m_w(i) = \m_w(i+1)$ if $i\notin I$ and $\m_w(i) = j$ if $i\in I$ and $j$ is such that either $(i,j)$ or $(j,i)$ is in $\Coess(w)$. The non-crossing condition implies that $\m_w$ is indeed an Hessenberg function and, if we enrich the set of Hessenberg functions with some extra datum (the datum where the inclusions change from $V_i\subset F_j$ to $F_i\subset V_j$) we can achieve a bijection, see \cite{gilboalapidabs}.
     
    We now turn our attention to characters of the Hecke algebra.  Each irreducible $\mathbb{C}$-representation of $S_n$ lifts to an irreducible $\mathbb{C}(q^{\frac{1}{2}})$-representation of $H_n$ (see \cite[Theorem 8.1.7]{GeckPf}). Hence, if $\chi^{\lambda}$ is the irreducible character of $S_n$ associated to the partition $\lambda\vdash n$ and, abusing notation, $\chi^{\lambda}$ is the corresponding character of $H_n$,  we can define the \emph{(dual) Frobenius character} of an element $a\in H_n$ by
    \[
    \ch(a):=\sum_{\lambda\vdash n} \chi^{\lambda}(a)s_{\lambda}(x)\in \mathbb{C}(q^{\frac{1}{2}})\otimes \Lambda,
    \]
    where $\Lambda$ is the algebra of symmetric functions in the variables $x= (x_1,\ldots, x_m,\ldots)$ and $s_{\lambda}(x)$ is the Schur symmetric function associated to the partition $\lambda$. For a graded $S_n$-module $L$ we also write $\ch(L)$ for its \emph{(graded) Frobenius character}.\par 
    
    In \cite[Lemma 1.1]{Haiman} Haiman proved that $\chi^{\lambda}(q^{\frac{\ell(w)}{2}}C'_w)$ is a symmetric unimodal  polynomial in $q$ with non-negative integer coefficients.  We note that \cite[Lemma 1.1]{Haiman} implies that $\ch(q^{\frac{\ell(w)}{2}}C'_w)$ is Schur-positive, in the sense that its coefficients in the Schur-basis are polynomials in $q$ with non-negative integer coefficients.

    Haiman also made some conjectures regarding positivity of the characters $\ch(q^{\frac{\ell(w)}{2}}C'_w)$ and relations between them. A symmetric function in $\mathbb{C}(q^{\frac{1}{2}})\otimes \Lambda$ is called $h$-positive if its coefficients in the complete homogeneous basis $\{h_{\lambda}\}$ are polynomials in $q$ with non-negative coefficients.
    
    \begin{conjecture}[Haiman]
    \label{conj:haimanhpos}
     For any $w\in S_n$ the (dual Frobenius) character $\ch(q^{\frac{\ell(w)}{2}}C'_w)$ of the Kazhdan-Lusztig element $C'_w$  is $h$-positive.
    \end{conjecture}

    If $\m$ is a Hessenberg function and $G_\m$ the associated indifference graph, we have by \cite{CHSS} (see also Corollary \ref{cor:hecke_csf} below) that the character $\ch(q^{\frac{\ell(w_\m)}{2}}C'_{w_\m})$ is the omega-dual of the chromatic quasisymmetric function of $G_\m$. In particular, Conjecture \ref{conj:haimanhpos} implies the Stanley-Stembridge conjecture on $e$-positivity of the chromatic symmetric function of indifference graphs of $3+1$ free posets (via results of Guay-Paquet, \cite{GPmodular}) and the Shareshian-Wachs generalization of the Stanley-Stembridge conjecture on $e$-positivity of the chromatic quasisymmetric function of indifference graphs.

    Haiman also made a conjecture about the relations between the characters $\ch(C'_w)$, namely, he that every character $\ch(C'_w)$ is a sum of characters of Kazdhan-Lusztig elements of codominant permutations.

    \begin{conjecture}[{\cite[Conjecture 3.1]{Haiman}  }]
    \label{conj:haimansingular}
    For any $w\in S_n$ there exist codominant permutations $w_1,\ldots, w_k$ such that
    \[
    \ch(C'_w)=\ch(C'_{w_1})+\ch(C'_{w_2})+\dots+ \ch(C'_{w_k})
    \]
    and\footnote{The condition on the Kazhdan-Lusztig polynomials is a consequence of the character equality.}
    \[
    P_{e,w}(q)=\sum_{1\leq i\leq k}q^{\frac{\ell(w)-\ell(w_i)}{2}}.
    \]
    \end{conjecture}
    Conjecture \ref{conj:haimansingular} restricts to the following statement when $w$ is smooth.
    
    \begin{conjecture}
    \label{conj:haimansmooth}
    If $w$ is a smooth permutation, there exists a single codominant permutation $w'$ such that 
    \[
    \ch(C'_w)=\ch(C'_{w'}).
    \]
    \end{conjecture}

    Haiman pointed out in \cite{Haiman} that Conjectures \ref{conj:haimansmooth} and \ref{conj:haimansingular} should ``reflect aspects of the geometry of the flag variety that cannot yet be understood using available geometric machinery".
    Conjecture \ref{conj:haimansmooth} was first proved combinatorially by Clearman-Hyatt-Shelton-Skandera in \cite{CHSS}. The purpose of this article is to provide a geometric proof of the same result, as well as a counter-example to Conjecture \ref{conj:haimansingular}.

\subsection{Results}
\label{sec:results}

   Let $X$ be an $n\times n$ matrix and $w$ be a permutation. The \emph{Lusztig variety} associated to $X$ and $w$ is the subvariety of the flag variety defined by
   \begin{equation}
       \h_w(X) : = \{ V_\bullet; XV_i\cap V_j \geq r_{i,j}(w)\text{ for }i,j=1,\ldots, n\}.
   \end{equation}
    When $X$ is regular semisimple (has distinct eigenvalues), the intersection homology $IH^*(\h_w(X))$ a natural $S_n$-module structure induced by the monodromy action of $\pi_1(GL_n^{rs},X)$ on $IH^*(\h_w(X))$. For $w$ a smooth permutation, so that $\h_w(X)$ is also smooth, this action can be explicitly characterized by a \emph{dot action} on $H^*(\h_w(X))$ (as in \cite{Tym08}). We have the following result due to Lusztig \cite{ChaShvV}, (see also \cite{AN_hecke}).
    \begin{theorem}[Lusztig]
    \label{thm:main}
           For any $w\in S_n$, we have $\ch(q^{\frac{\ell(w)}{2}}C'_w)=\ch(IH^*(\h_w(X)))$.
    \end{theorem}
    
    \noindent In Section \ref{sec:haimanconj} we will prove the following:

     \begin{theorem}
    \label{thm:mainhaimanconj}
           Let $X \in SL_n(\mathbb{C})$ be regular semi-simple and $w\in S_n$ smooth. Then there exists a codominant permutation $w'$ such that $H^*(\h_w(X))$ and $H^*(\h_{w'}(X))$ are isomorphic as $S_n$-modules. In particular, $\ch(C'_w)=\ch(C'_{w'})$.
         \end{theorem}
         The main idea is to see that both $\h_w(X)$ and $\h_{w'}(X)$ are smooth GKM spaces, and hence their cohomologies are described by their moment graphs. Since the moment graph of $\h_w(X)$ only depends on the transpositions which are smaller than $w$ in the Bruhat order, it suffices to see that there exists a codominant permutation whose set of smaller transpositions is equal to that of $w$. In fact, these transpositions are precisely the transpositions $(i,j)$ such that $i < j \leq \m_w(i)$ (see, for example, \cite{gilboalapidabs}).\par
       
         If $w$ and $w'$ are Coxeter elements, a stronger result holds, and we actually have that $\h_{w}(X)$ is isomorphic to $\h_{w'}(X)$ whenever $X$ is regular semisimple (see \cite[Example 1.23]{AN_hecke}). Although for Coxeter elements, Conjecture \ref{conj:haimansmooth} is a consequence of \cite[Proposition 4.2]{Haiman}. We note that our proof of Theorem \ref{thm:mainhaimanconj} only proves the isomorphisms of cohomology groups and not of varieties (see Conjecture \ref{conj:wcod}).
         
         Concerning singular permutations, we have the following theorems.
        
          \begin{theorem}
          \label{thm:relation}
          Let $w\in S_n$ be a singular permutation and $s$ a simple transposition such that $ws$ is smooth and $sws<w$. Then
          \[
          \ch(C'_w)=(q^{-1/2}+q^{1/2})\ch(C'_{ws}).
          \]
          The analogous equality holds if $sw$ is smooth. Geometrically, if $w$ and $s$ satisfy the above conditions and $X$ is regular semisimple, then $\h_w(X)$ and $\h_{ws}(X)$ fit into the following  diagram
         \[
          \begin{tikzcd}
            & \h_{ws}(X) \ar[d,"g"]\\
          \h_w(X)\ar[r,"f"] & \mathcal{Z}
          \end{tikzcd}          
          \]
          where $f$ is a $\mathbb{P}^1$-bundle and $g$ is small.
          \end{theorem}
          
           Theorem \ref{thm:relation} is a direct consequence of Corollary \ref{cor:relation}, Lemma \ref{lem:P1bundleinj} and Proposition \ref{prop:semismall}. These results also apply when $w$ is smooth, in which case we recover the so-called \emph{modular law} for the chromatic quasisymmetric function of indifference graphs (see \cite{AN}) and provide a geometric interpretation of it in Example \ref{exa:modularlawgeo} (See also \cite{dCPL} and \cite{PrecupSommers}). The modular law also appears in other symmetric functions associated to indifference graphs, such as the LLT-polynomials (\cite{Lee}) and the symmetric function of increasing forests (\cite{ANtree}).

         \begin{theorem}[Counter-example to Conjecture \ref{conj:haimansingular}]
         \label{thm:counter}
          Let $w=62754381 \in S_8$. Then $P_{e,w}(q)=1+q$ and there do not exist codominant permutations $w_0$, $w_2$ such that 
          \[
          \ch(C'_w)=\ch(C'_{w_0})+\ch(C'_{w_2}).
          \]
         \end{theorem}
        \begin{proof}
        
        Set $s=(1,2)$. Then $sws = 16754382 < w$. Moreover, $ws = 26754381 = w_{\m_1}$, where $\m_1=(2,6,7,7,7,7,8,8)$ is a Hessenberg function. In particular, $ws$ is codominant, hence smooth, so that $P_{e,w}(q)=1+q$. Assume that there exist codominant permutations $w_0$ and $w_2$ such that
          \[
          (q^{-\frac{1}{2}}+q^{\frac{1}{2}})\ch(C'_{ws})=\ch(C'_{w_0})+\ch(C'_{w_2}).
          \]
        Then there exist Hessenberg functions $\m_0$ and $\m_2$ such that (recalling $w_{\m_1}=ws$)
          \begin{equation}
          \label{eq:modularm}    
          (1+q)\csf_q(G_{\m_1})=\csf_q(G_{\m_2})+q\csf_q(G_{m_0}).
          \end{equation}
         But by exhaustive search using the Algorithm in \cite{AN}, there do not exist $\m_0$ and $\m_2$ satisfying this condition, which finishes the proof.
         
         \begin{figure}[h]
        \begin{tikzpicture}
\begin{scope}[scale=0.3]
\draw[help lines] (0,0) grid +(8,8);
\dyckpath{0,0}{1,1,0,1,1,1,1,0,1,0,0,0,0,1,0,0}
\node at  (0.5,5.5) [shape=circle, fill=black, inner sep=1pt] {};
\node at  (1.5,1.5) [shape=circle, fill=black, inner sep=1pt] {};
\node at  (2.5,6.5) [shape=circle, fill=black, inner sep=1pt] {};
\node at  (3.5,4.5) [shape=circle, fill=black, inner sep=1pt] {};
\node at  (4.5,3.5) [shape=circle, fill=black, inner sep=1pt] {};
\node at  (5.5,2.5) [shape=circle, fill=black, inner sep=1pt] {};
\node at  (6.5,7.5) [shape=circle, fill=black, inner sep=1pt] {};
\node at  (7.5,0.5) [shape=circle, fill=black, inner sep=1pt] {};
\end{scope}
\end{tikzpicture}
    \caption{The graphical representation of the Dyck path associated to the Hessenberg function $\m_1=(2,4,5,5,6,6)$ and of the permutation $w = 62754381$.}
    \label{fig:cod_permutation8}
    \end{figure}
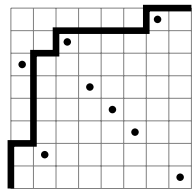
    
         \end{proof}
        
  In view of Theorems \ref{thm:relation} and \ref{thm:counter}, we propose a weaker version of Conjecture \ref{conj:haimansingular}:
\begin{conjecture}
\label{conj:haimanfixed}
For each permutation $w\in S_n$ there exists codominant permutations $w_1,\ldots, w_k\in S_n$ such that $\ch(q^{\frac{\ell(w)}{2}}C'_w)$ is a combination of $\ch(q^{\frac{\ell(w_i)}{2}}C'_{w_i})$ with coefficients in $\mathbb{N}[q]$.
\end{conjecture}

\section{Proof of Theorem \ref{thm:mainhaimanconj}}
\label{sec:haimanconj}
   We begin by recalling some properties of GKM-spaces (see \cite{GKM}). A \emph{GKM-space}, is a smooth projective variety $\mathcal{X}$ with an action of a torus $T$ such that the number of fixed points and the number of $1$-dimensional orbits are finite. The equivariant cohomology $H_T^*(\mathcal{X})$ is then encoded in a combinatorial object called the \emph{moment graph} of $\mathcal{X}$. The vertices of the moment graph are the fixed points, while the edges are the $1$-dimensional orbits, each of which has exactly two fixed points on its closure.  \par
     
     If $X$ is an $n\times n$ diagonal regular semisimple matrix, the torus $T\cong (\mathbb{C}^*)^n$ of diagonal matrices acts on the variety $\h_w(X)$. When $w$ is smooth, this variety is a GKM-space because the action is a restriction of that of $T$ on the whole flag variety, where the number of fixed points and $1$-dimensional orbits are indeed finite. Moreover, the moment graph also encodes the action of $S_n$ on the equivariant cohomology group $H_T^*(\h_w(X))$ (induced by the monodromy action of $\pi_1(GL_n^{rs},X)$), usually called the \emph{dot action} (see \cite{Tym08}). In particular, if $w$ and $w'$ are smooth permutations and $\h_w(X)$ and $\h_{w'}(X)$ have the same moment graph, then $\ch(H^*(\h_w(X)))=\ch(H^*(\h_{w'}(X)))$. \par
  
  Since $\h_w(X)$ is a $T$-invariant subvariety of $\flag$, we have that the moment graph of $\h_w(X)$ is a subgraph of the moment graph of the flag variety $\flag$. We briefly recall the moment graph of $\flag$ (see \cite{Carrell} and \cite[Proposition 2.1]{Tym08}). The fixed points in $\flag$ are indexed by permutations $w\in S_n$ (in fact, they are equal to $\h_e(X)$ for $X$ a regular semisimple diagonal matrix). To see this, it is enough to see that a flag $V_\bullet$ is fixed by $T$ if and only if each $V_i$ is generated by eigenvectors of $T$. However the eigenvectors of $T$ are precisely the canonical basis vectors $e_1,\ldots, e_n$, so there exists $w\in S_n$ such that $V_i=\langle e_{w(1)}, \ldots, e_{w(n)}\rangle$.
  
  The $1$-dimensional orbits are associated to tuples $(w_1,w_2, t)$, where $w_1,w_2\in S_n$ (corresponding to fixed points) with $\ell(w_1)<\ell(w_2)$ and $t$ is a transposition satisfying $w_1=w_2t$. Then the orbit can be described as follows: Write $t=(i j)$ with $i<j$ and define $v_{i}=e_{w_2(i)}+ce_{w_2(j)}$ for $c\in \mathbb{C}^*$. When variying $c\in \mathbb{C}^*$, the flags $V_\bullet^c$ given by $V_k^c=\langle e_{w_2(1)}\ldots e_{w_2(i-1)}, v_{i}, e_{w_2(i+1)},\ldots, e_{w_2(k)}\rangle$ determine the $1$-dimensional orbit given by $(w_1,w_2,t)$. In fact when $c$ goes to $0$, the limit of $V_\bullet^c$ is the flag induced by $w_2$, while when $c$ goes to infinity, the limit of $V_\bullet^c$ is $V_{w_1}$.
  So the $1$-dimensional orbit associated to $(w_1,w_2,t)$ connects the fixed points corresponding to $w_1$ and $w_2$. \par
  
  To describe the moment graph of $\h_w(X)$ it is enough to see which fixed points and $1$-dimensional orbits are contained in $\h_w(X)$. Since $\h_e(X)\subset \h_w(X)$, we have that all fixed points of $\flag$ belong in $\h_w(X)$. We claim the following.
  \begin{lemma}
  \label{lem:1-orbits}
  The $1$-dimensional orbit associated to $(w_1,w_2,t)$ is contained in $\h_w(X)$ if and only if the transposition $t$ is smaller than $w$ in the Bruhat order of $S_n$.
  \end{lemma}
  \begin{proof}
  Consider the flag $V_\bullet^c$ in the $1$-dimensional orbit $(w_1,w_2,t)$. An easy computation shows that $XV^c_{\ell}\cap V^c_k =r_{\ell, k}(t)$. In particular, $V_\bullet^c\in \h_t(X)^\circ$. Since $\h_w(X)= \bigsqcup_{z\leq w}\h_z(X)^\circ$, we have that $V_\bullet^c\in \h_w(X)$ if and only if $t\leq w$.
  \end{proof}
  
  \begin{lemma}
  \label{lem:t<w}
   Let $w$ be a smooth permutations and $\m$ its associated Hessenberg function. A transposition $t=(ij)$ with $i<j$ is smaller that or equal to $w$ in the Bruhat order of $S_n$ if and only if $j\leq \m(i)$.
  \end{lemma}
  \begin{proof}
  This is contained in \cite[Theorem 5.1]{gilboalapidabs}. One can see this geometrically from the characterization of smooth Schubert varieties. Consider the pair $(V_\bullet, F_\bullet)$ where $V_\bullet$ is induce by the matrix $(e_1,\ldots, e_{i-1}, e_j, e_{i+1},\ldots,e_{j-1}, e_i, e_{j+1},\ldots, e_n)$ and $F_\bullet$ is induced by the identity matrix $(e_1,\ldots, e_n)$. Then we have $V_i\subset F_j$ and $F_i\subset V_j$, but $V_i\not\subset F_{j-1}$ and $F_i\not\subset V_{j-1}$. In particular, we have that $(V,F)\in \Omega_w$ if and only if $j\leq \m(i)$. Since $(V,F)\in \Omega_t^\circ$, the result holds.
  \end{proof}

  \begin{proof}[Proof of Theorem \ref{thm:mainhaimanconj}] 
   Let $w'$ be the codominant permutation associated to the Hessenberg function $\m$ associated to $w$. By Lemmas \ref{lem:1-orbits} and \ref{lem:t<w}, the moment graphs of $\h_w(X)$ and $\h_{w'}(X)$ are equal, and since the dot action only depends on the moment graph, $\ch(H^*(\h_w(X)))=\ch(H^*(\h_{w'}(X)))$. By Theorem \ref{thm:main} we have the result.
\end{proof}

\section{Proof of Theorem \ref{thm:relation}}

To prove Theorem \ref{thm:relation} we need a few algebraic results about Hecke algebras and singular permutations. Let $w\in S_n$ be a permutation and $s$ a simple transposition. Assume that $sw<w<ws$. Then by the multiplication rule of Kazhdan-Lusztig elements of the Hecke algebra (see \cite[Equation 8.8]{Haiman}) we have
\begin{align*}
    C'_wC'_s=&C'_{ws}+\sum_{\substack{z\leq w\\ zs<z}}\mu(z,w)C'_z,\\
    C'_sC'_w=&(q^{-\frac{1}{2}}+q^{\frac{1}{2}})C'_w,
\end{align*}
where $\mu(z,w)$ is the coefficient of $q^{\frac{\ell(w)-\ell(z)-1}{2}}$ in the Kazhdan-Lusztig polynomial $P_{z,w}(q)$. Since $\chi^\lambda(C'_wC'_s)=\chi^\lambda(C'_sC'_w)$ for every partition $\lambda\vdash n$, we have that 
\begin{equation}
    \label{eq:modular}
    \ch((q^{-\frac{1}{2}}+q^{\frac{1}{2}})C'_w)=\ch(C'_{ws})+\sum_{\substack{z\leq w\\ zs<z}}\mu(z,w)\ch(C'_z).
\end{equation}
If $w$ is smooth, then $\mu(z,w)=0$ except for the permutations $z$ such that $z\leq w$ and $\ell(z) = \ell(w)-1$, and in this case $\mu(z,w)=1$. To simplify notation, we will write $z\lessdot w$ to mean that $z\leq w$ and $\ell(z)=\ell(w)-1$. We will see below that if $w$ is smooth and satisfies $sw<w<ws$ for some simple reflection $s$, then there exists at most one permutation $z$ satisfying  $z\lessdot w$ and $zs<z$.

\begin{proposition}
\label{prop:ws}
Let $w\in S_n$ be a smooth permutation and $s$ a simple reflection such that $sw<w<ws$. Then one of the following holds:
\begin{enumerate}
    \item The permutation $ws$ is smooth and there exists precisely one $z\lessdot w$ such that $zs<z$. Moreover, $z$ is smooth.
    \item The permutation $ws$ is singular and there does not exists any $z\lessdot w$ such that $zs<z$.
\end{enumerate}
\end{proposition}
\begin{proof}
We first prove that there exists at most one $z\lessdot w$ such that $zs<z$. Write $s=(l,l+1)$ and assume that $z\in S_n$ is a permutation satisfying $z\lessdot w$ and $zs<s$. Since $z\lessdot w$ (which means that $\ell(z)=\ell(w)-1$), we have that there exist $i_1,i_2$ such that 
\begin{itemize}
    \item $1\leq i_1<i_2\leq n$, 
    \item $z(j)=w(j)$ for every $j\in[n]\setminus\{i_1,i_2\}$,
    \item $z(i_k)=w(i_{3-k})$,
    \item $w(i_1)>w(i_2)$,
    \item for every $i_1<j<i_2$ we have that either $w(j)<w(i_2)$ or $w(j)>w(i_1)$. 
    \end{itemize}
    Since $ws>w$ and $zs<z$, we have $w(l)<w(l+1)$ and $z(l)>z(l+1)$. Hence either $i_1=l+1$ or $i_2=l$. \par

If $i_1=l+1$, we have
\begin{equation}
\label{eq:i1l}
\begin{aligned}
      &w(j)<w(i_2)\text{ or }w(j)>w(i_1)=w(l+1)\text{ for every }i_1<j<i_2,\\
      &w(l+1)=w(i_1)>w(l)>w(i_2).
\end{aligned}
\end{equation}

On the other hand, if $i_2=l$, we have
\begin{equation}
\label{eq:i2l}
\begin{aligned}
      &w(j)<w(i_2)=w(l)\text{ or }w(j)>w(i_1)\text{ for every }i_1<j<i_2,\\
      &w(i_1)>w(l+1)>w(i_2)=w(l).
\end{aligned}
\end{equation}
See figures \ref{fig:1} and \ref{fig:2} below for a depiction of these conditions.

\begin{figure}[h]

 \centering
 \begin{tikzpicture}
\begin{scope}[scale=0.3]
\draw[help lines] (0,0) grid +(8,8);

\dyckpath{3,2}{1,1,1,1,0,0,0};
\dyckpath{3,2}{0,0,0,1,1,1,1};
\node at  (1.5,3.5) [shape=circle, fill=black, inner sep=1pt] {};
\node at  (2.5,6.5) [shape=circle, fill=black, inner sep=1pt] {};
\node at  (6.5,1.5) [shape=circle, fill=black, inner sep=1pt] {};

\node at (1.5, -0.6) {$l$};
\node at (2.5, -0.6) {$i_1$};
\node at (6.5, -0.6) {$i_2$};
\end{scope}
\end{tikzpicture} 
    \caption{The relative position of $w(l)$, $w(i_1)$, and $w(i_2)$ given by Equation \eqref{eq:i1l}. Note that we can not have any dots inside the box. }
    \label{fig:1}
\begin{tikzpicture}
\begin{scope}[scale=0.3]
\draw[help lines] (0,0) grid +(8,8);

\dyckpath{2,2}{1,1,1,1,0,0,0};
\dyckpath{2,2}{0,0,0,1,1,1,1};
\node at  (6.5,3.5) [shape=circle, fill=black, inner sep=1pt] {};
\node at  (1.5,6.5) [shape=circle, fill=black, inner sep=1pt] {};
\node at  (5.5,1.5) [shape=circle, fill=black, inner sep=1pt] {};

\node at (6.5, -1.7) {$l+1$};
\node at (1.5, -0.6) {$i_1$};
\node at (5.5, -0.6) {$i_2$};
\end{scope}
\end{tikzpicture} 
    \caption{The relative position of $w(l+1)$, $w(i_1)$, and $w(i_2)$ given by Equation \eqref{eq:i2l}. Note that we can not have any dots inside the box. }
    \label{fig:2}    
    \end{figure}

  Assume that there exist two distinct permutations $z,z'$ satisfying the conditions above, and let $i_1,i_2$ and $i_1',i_2'$ be as above for $z$ and $z'$, respectively. We now compare the relative position of $i_1,i_2,i_1',i_2'$.\par
  \begin{itemize}
      \item Case 1. Assume that $i_2=i'_2=l$ and $i_1<i'_1$ (the case $i_1<i_1'$ being analogous). By Equation \eqref{eq:i2l}, we have that $w(i_1)>w(l+1)>w(l)$, $w(i_1')>w(l+1)>w(l)$. Since $i_1<i_i'<i_2$ and $w(i_1')>w(l)$,  we have  $w(i'_1)>w(i_1)$ (again, by Equation \eqref{eq:i2l}). Hence $w(i'_1)>w(i_1)>w(\ell+1)>w(\ell)$ and this is a $3412$ pattern on $w$, which is a contradiction with the smoothness of $w$. See Figure \ref{fig:case_1}.
      \begin{figure}[h]
\begin{tikzpicture}
\begin{scope}[scale=0.3]
\draw[help lines] (0,0) grid +(10,10);

\dyckpath{2,2}{1,1,1,1,0,0,0,0,0};
\dyckpath{2,2}{0,0,0,0,0,1,1,1,1};
\node at  (8.5,3.5) [shape=circle, fill=black, inner sep=1pt] {};
\node at  (1.5,6.5) [shape=circle, fill=black, inner sep=1pt] {};
\node at  (7.5,1.5) [shape=circle, fill=black, inner sep=1pt] {};
\node at  (3.5,8.5) [shape=circle, fill=black, inner sep=1pt] {};

\node at (8.5, -1.7) {$l+1$};
\node at (1.5, -0.6) {$i_1$};
\node at (3.5, -0.6) {$i_1'$};
\node at (7.5, -0.6) {$i_2$};
\end{scope}
\end{tikzpicture} 
    \caption{The relative position of $w(i_1)$, $w(i_1')$, $w(i_2)$, and $w(l+1)$. Note that $w(i_1')$ must be outside the box, and hence $w(i_1')>w(i_1)$, we have a $3412$ pattern on $w$. }
    \label{fig:case_1}   
    \end{figure}
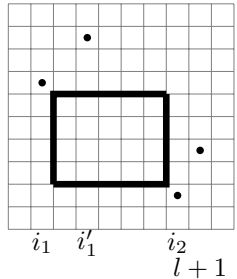
      
      \item Case 2. Assume that $i_1=i'_1=l+1$. This case is analogous to the previous one (just replace Equation \eqref{eq:i2l} with Equation \eqref{eq:i1l}).

      \item Case 3. Assume that $i_2=l$ and $i'_1=l+1$. In this case, we have that $i_1<i_2=l<i_1'=l+1<i_2'$. By Equations \eqref{eq:i2l} and \eqref{eq:i1l},  $w(l+1)> w(l)>w(i_2')$ and $w(i_1)>w(l+1)>w(l)$, so $w(i_1)>w(l+1)>w(l)>w(i'_2)$, which is a $4231$ pattern on $w$, contradicting the smoothness of $w$. See Figure \ref{fig:case_3}
            \begin{figure}[h]
\begin{tikzpicture}
\begin{scope}[scale=0.3]
\draw[help lines] (0,0) grid +(11,11);

\dyckpath{2,4}{1,1,1,1,0,0,0};
\dyckpath{2,4}{0,0,0,1,1,1,1};

\node at  (1.5,8.5) [shape=circle, fill=black, inner sep=1pt] {};
\node at  (5.5,3.5) [shape=circle, fill=black, inner sep=1pt] {};

\dyckpath{7,2}{1,1,1,0,0};
\dyckpath{7,2}{0,0,1,1,1};

\node at  (6.5,5.5) [shape=circle, fill=black, inner sep=1pt] {};
\node at  (9.5,1.5) [shape=circle, fill=black, inner sep=1pt] {};

\node at (6.5, -0.6) {$i_1'$};
\node at (9.5, -0.6) {$i_2'$};
\node at (1.5, -0.6) {$i_1$};
\node at (5.5, -0.6) {$i_2$};
\end{scope}
\end{tikzpicture} 
    \caption{The relative position of $w(i_1)$, $w(i_2)$, $w(i_1')$, and $w(i_2)$.  Note that we have a $4231$ pattern on $w$. }
    \label{fig:case_3}  
    \end{figure}
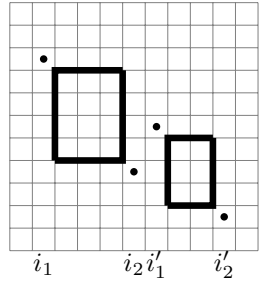
  \end{itemize}
  Similar considerations also prove that if $z$ exists, it must  be smooth.

   We now prove that if $ws$ is singular, there exists no $z\lessdot w$ with $zs<z$. Since $ws$ is singular, there exist $j_1<j_2<j_3<j_4$ forming a $4231$ or $3412$ pattern in $ws$. Since $w$ is smooth, $\{l,l+1\}\subset \{j_1,j_2,j_3,j_4\}$. Since $w(l)<w(l+1)$ we have three cases.\par
   \begin{itemize}
       \item Case 1. Assume that we have a $4231$ pattern in $ws$ with $j_1=l,j_2=l+1$. Then $j_1,j_2,j_3,j_4$ induces a $2431$ pattern on $w$ with $j_1=l,j_2=l+1$. Let us assume that there exists $i_1<i_2:=l=j_1$ satisfying Equation \eqref{eq:i2l}. Then $w(i_1)>w(l+1)$ and $i_1,j_1,\ldots, j_4$  induces a $52431$ pattern on $w$, which contains a $4231$ pattern, and this is a contradiction. Let us assume that there exists $l+1=j_2=:i_1<i_2$ satisfying Equation \eqref{eq:i1l}. Then $w(i_2)<w(l)$ and for every $l+1<k<i_2$ we have either $w(k)>w(l+1)$ or $w(k)<w(i_2)$. Then $i_2<j_3$ since $w(i_2)<w(l)<w(j_3)<w(l+1)$. This means that $w$ contains either a $35241$ or a $35412$ pattern, but the first has a $4231$ pattern, while the second has a $3412$ pattern, which again contradicts the smoothness of $w$. 
       
       \item Case 2. Assume that we have a $4231$ pattern in $ws$ with $j_3=l, j_4=l+1$. Then we have a $4213$ pattern on $w$, and the argument is similar as above.\par
     
     \item Case 3. Assume that we have a $3412$ pattern in $ws$ with $j_2=l, j_3=l+1$, so that $j_1,j_2,j_3,j_4$ induces a $3142$ pattern on $w$ with $j_2=\ell,j_3=\ell+1$. Let us assume there exists $i_1<i_2:=l=j_2$ satisfying Equation \eqref{eq:i2l}. Then $w(i_1)>w(l+1)$, and for every $i_1<k<l$ we have either $w(k)>w(i_1)$ or $w(k)<w(l)$. Then $i_1>j_1$ and we have a $35142$ pattern on $w$, a contradiction. Let us assume that there exist $l+1=j_3=:i_1<i_2$ satisfying Equation \eqref{eq:i1l}. Then $w(i_2)<w(l)$ and for every $l+1<k<i_2$ we have either $w(k)<w(i_2)$ or $w(k)>w(l+1)$, so that $i_2<j_4$ and we have a $42513$ pattern on $w$, also a contradiction.\par
   \end{itemize}
   
    Finally, we will prove that if there is no $z\lessdot w$ with $zs<z$, then $ws$ is singular. First, assume that there exists $i<l$ such that $w(i)>w(l+1)$ and consider the greatest possible such $i$. If $z=w\cdot (i,l)$, then $zs<z$ and $z<w$. This means that $z<<w$, and that is equivalent to the existence of $i<j<l$ with $w(i)>w(j)>w(l)$. Since $i$ is the greatest $i<l$ with $w(i)>w(l+1)$, we have that $w(i)>w(l+1)>w(j)>w(l)$, which implies that $i,j,l,l+1$ induces a $4213$ pattern on $w$ and hence a $4231$ pattern on $ws$. If there exists $i>l+1$ with $w(i)<w(l)$, the argument is the same. \par
    Therefore, let us assume that $w(i)<w(l+1)$ for every $i<l$ and $w(i)>w(l)$ for every $i>l+1$. In particular, we have that $w^{-1}(j)<l$ for every $j<w(l)$. Let $k$ be the maximum of $\{w(i)\}_{i\leq l}$, and note that $w(l)\leq k < w(\ell+1)$.  Assume that there exists $ j<k$ with $w^{-1}(j)>l+1$. By the argument above, we have that $j>w(l)$ (and hence $k>w(l)$), so $w^{-1}(k)<l<l+1<w^{-1}(j)$ and $w(l+1)>k>j>w(l)$, which implies that $w^{-1}(k),l,l+1,w^{-1}(j)$ induces a $3142$ pattern on $w$, and hence a $4231$ pattern on $ws$. On the other hand, if $w^{-1}(j)\leq l$ for every $j\leq k$, then $\{w(1),\ldots, w(l)\}=\{1,\ldots, k\}$, and in particular $k=l$. But then $(l,l+1)w>w$, a contradiction since $sw<w$ by hypothesis. This finishes the proof.\par

\end{proof}

We have the following direct corollary.
\begin{corollary}
\label{cor:relation}
Let $w$ be a smooth permutation and $s$ a simple transposition such that $ws>w>sw$.
\begin{enumerate}
    \item If $ws$ is smooth and $z$ is the only permutation $z\lessdot w$ with $zs<z$, then $(q^{-\frac{1}{2}}+q^{\frac{1}{2}})\ch(C'_{w})=\ch( C'_{ws})+\ch(C'_{z})$.
    \item If $ws$ is singular, then $(q^{-\frac{1}{2}}+q^{\frac{1}{2}})\ch(C'_w)=\ch(C'_{ws})$.
\end{enumerate}
\end{corollary}
\begin{proof}
Follows directly from Equation \eqref{eq:modular} and Proposition \ref{prop:ws}.
\end{proof}

 Corollary \ref{cor:relation} has a geometric interpretation. Let $w$ and $s$ be as in Corollary \ref{cor:relation}, and let $\mathcal{P}_{s}$ be the partial flag variety associated to $s$, that is, if $s=(l,l+1)$ then 
 \[
 \mathcal{P}_s=\{V_1\subset V_2\subset \ldots V_{l-1}\subset V_{l+1}\subset\ldots \subset V_{n}=\mathbb{C}^n; \dim_{\mathbb{C}}(V_i)=i\}.
 \]
 Using the algebraic group notation, we write $G=GL_n$ and $B$ for the Borel subgroup of $G$ of uppertriangular matrices. For each permutation $w\in S_n$ let $\dw$ denote the associated permutation matrix $\dw\in G$. We write $P_s$ for the parabolic subgroup associated to $s$, that is, $P_s = B\sqcup B\ds B$, so that $\mathcal{P}_s=G/P_s$. In this notation, the Lusztig varieties are given by $\h_w(X)^\circ=\{gB; g^{-1}Xg\in B\dw B\}$.
 
\begin{lemma}
\label{lem:P1bundleinj}
Let $w\in S_n$ be a permutation, $s$ a simple transposition, and $X$ a regular semisimple $n\times n$ matrix. Then
\begin{enumerate}
    \item If $sw<w$ and $ws<w$, then the forgetful map $\h_w(X)\to \mathcal{P}_s$ is a $\mathbb{P}^1$-bundle over its image.
    \item If $ws\neq sw$ and either $w<ws$ or $w<sw$, then the forgetful map $\h_w^\circ(X)\to \mathcal{P}_s$ is injective.
\end{enumerate}
\end{lemma} 
 \begin{proof}
 We begin with item (1). For $s=(l,l+1)$ the hypothesis is equivalent to $w(l)>w(l+1)$ and $w^{-1}(l)>w^{-1}(l+1)$, and in particular, the coessential set of $w$
\[
\Coess(w):=\{(a,b); w(a)\leq b<w(a+1), w^{-1}(b)\leq a< w^{-1}(b+1)\}
\]
does not contains any pair $(a,b)$ with either $a=l$ or $b=l$. This means that the conditions involving $\dim(XV_l\cap V_b)$ and $\dim(XV_a\cap V_l)$ are redundant in $\h_{w}(X)$, hence $V_l$ can be chosen arbitrarily.

Let us prove item (2). Since $\h_w^\circ(X)=\{ gB; g^{-1}Xg\in B\dw B\}$, to prove that the map $\h_w^\circ(X)\to \mathcal{P}_s$ is injective it suffices to prove that there do not exist $g_1B$ and $g_2B$ distinct such that $g_1^{-1}Xg_1\in B\dw B$, $g_2^{-1}Xg_2\in B\dw B$, and $g_1\in g_2P_s$. Assume by way of contradiction that such a pair $g_1,g_2$ exists. Since $P_s=B\cup B\ds B$ and $g_1B\neq g_2B$, we have that $g_1\in g_2B\ds B$, in particular $g_1=g_2 b_1\ds b_2 $ for some $b_1,b_2\in B$. Therefore
\begin{align*}
    g_2^{-1}Xg_2&\in B\dw B,\\
    b_2^{-1}\ds b_1^{-1}g_2^{-1}Xg_2 b_1\ds b_2 &\in B\dw B. 
\end{align*}
Since $b_1,b_2\in B$, we have
\begin{align*}
    b_1^{-1}g_2^{-1}Xg_2b_1&\in B\dw B,\\
    b_1^{-1}g_2^{-1}Xg_2 b_1 &\in \ds B\dw B \ds.
\end{align*}
This means that $B\dw B\cap \ds B\dw B\ds\neq \emptyset$. Let us assume, without loss of generality, that $sw< w<ws$. Then by \cite[proof of Lemma 11.14]{MalleTesterman} 
\begin{align*}
    B\dw B \ds\subset B \dw B \cdot B\ds B = B\dw \ds B,
\end{align*}
and by \cite[Lemma 11.14]{MalleTesterman}
\[
\ds B\dw\ds B\subset B\ds B \cdot B\dw\ds B\subset B\dw\ds B\cup B\ds\dw\ds B.
\]
Since $sws\neq w$ (otherwise, $ws=sw$), we have
\[
B\dw B\cap (B\dw\ds B\cup B\ds\dw\ds B)\neq \emptyset,
\]
which is a contradiction of the Bruhat decomposition of $G$.
 \end{proof}

Let $X$ be a regular matrix, $w\in S_n$ an irreducible permutation, that is, a permutation that is not contained in any proper Young subgroup, and $s$ a simple transposition satisfying the conditions in Corollary \ref{cor:relation}. Consider the forgetful map $\h_{ws}(X)\to \mathcal{P}_s$ and let $\mathcal{Z}$ be the image. By \cite[Corollary 8.6]{AN_hecke}, $\h_{ws}(X)$ and $\h_w(X)$ are irreducible, and so $\mathcal{Z}$ is as well. By Lemma \ref{lem:P1bundleinj}, the map $\h_{ws}(X)\to \mathcal{Z}$ is a $\mathbb{P}^1$-bundle, while the map $\h_w^\circ(X) \to \mathcal{P}_s$ is injective. Since $\h_w(X)\subset \h_{ws}(X)$ ($w<ws$), the image of $\h_w(X)$ is contained in $\mathcal{Z}$. Since $\h_w^\circ(X)\to \mathcal{Z}$ is injective and the dimensions agree, $\h_w(X)\to \mathcal{Z}$ is birational. Let $z\in S_n$ be the permutation such that $z\lessdot w$ and $zs<z$. Then we have:

\begin{proposition}
\label{prop:semismall}
The map $\h_w(X)\to \mathcal{Z}$ is semismall and the preimage of the relevant locus is precisely $\h_z(X)$ (if $z$ exists).
\end{proposition}
\begin{proof}
The fact that $\h_w(X)\to \mathcal{Z}$ is semismall follows from the fact that the map is birational and its fibers have dimension at most one (since they are contained in those of $\h_{ws}(X)\to \mathcal{Z}$). We have that $\h_w(X)=\h_w^{\circ}(X)\cup\bigcup_{z'\lessdot w} \h_{z'}(X)$, where $\h_{z'}(X)$ has codimension one in $\h_w(X)$. We claim that the images of $\h_w^\circ(X)$ and $\h_{z'}(X)$ are disjoint. Assume for contradiction that there exist $g_1B$ and $g_2B$ such that $g_1^{-1}Xg_1\in B\dw B$, $g_2^{-1}Xg_2 \in \overline{B\dz'B}$ and $g_1P_s=g_2P_s$. Arguing as in the proof of Lemma \ref{lem:P1bundleinj}, we have
 \begin{equation}
     \label{eq:BzB}
  \overline{B\dz'B}\cap (B\dw\ds B\cup B\ds\dw\ds B)\neq \emptyset.
 \end{equation}
 However, $\ell(ws)=\ell(w)+1$, $\ell(sws)=\ell(w)$, and $\ell(z)=\ell(w)-1$, and $\overline{B\dz'B}=\bigcup_{z''\leq z'} B\dz'' B$. By the Bruhat decomposition, Equation \eqref{eq:BzB} is a contradiction.

 Moreover, since the fibers have dimension at most one, the preimage of the relevant locus has codimension one in $\h_w(X)$. By the discussion above, this preimage must be a union of $\h_{z'}(X)$ for some $z'\lessdot w$. By the lifiting property \cite[Proposition 2.2.7]{Brenti}, either $sz'<z'$ or $z'=sw$. If $z'=sw$, then $z'=sw<sws=zs'$ and $z's=sws\neq w=sz'$, so by Lemma \ref{lem:P1bundleinj} $\h_{z'}^{\circ}(X)\to \mathcal{Z}$ is injective, and hence $\h_{z'}^\circ(X)$ is not contained in the preimage of the relevant locus. If $sz'<z'$ and $z'<z's$, then $sz'\neq z's$, so by Lemma \ref{lem:P1bundleinj} $\h_{z'}^{\circ}(X)\to \mathcal{Z}$ is injective, and hence $\h_{z'}^\circ(X)$ is not contained in the preimage of the relevant locus. Finally, if $sz'<z'$ and $z's< z'$, then $z'=z$, so by Lemma \ref{lem:P1bundleinj} $\h_{z'}(X)\to \mathcal{Z}$ is $\mathbb{P}^1$-bundle over its image, and hence $\h_{z'}(X)$ is contained in the preimage of the relevant locus. Since the preimage of the relavant locus has codimension one, it is precisely $\h_{z'}(X)$.
\end{proof}

By the decomposition theorem (we set $\mathcal{Z}_1$ as the image of $\h_z(C)$ if $z$ exsits), $IH^{*}(\h_{ws}(X))=IH^{*}(\mathcal{Z})\otimes (\mathbb{C}\oplus \mathbb{C}[-2])$,  $H^*(\h_w(X))=IH^*(\mathcal{Z})\otimes IH^*(\mathcal{Z}_1)[-2] $ and $IH^*(\h_{z}(X))=IH^{*}(\mathcal{Z_1})\otimes (\mathbb{C}\oplus \mathbb{C}[-2])$. Then
\begin{align*}
    \ch(IH^{*}(\h_{ws}(X)))&=(1+q)\ch(IH^{*}(\mathcal{Z})),\\
    \ch(H^*(\h_w(X)))&=\ch(IH^*(\mathcal{Z}))+q\ch(IH^*(\mathcal{Z}_1)),\\
    \ch(IH^*(\h_{z}(X)))&=(1+q)\ch(IH^{*}(\mathcal{Z_1})),
\end{align*}
which implies
\[
(1+q)\ch(H^*(\h_w(X)))=\ch(IH^{*}(\h_{ws}(X)))+q\ch(IH^*(\h_{z}(X))).
\]
This, in turn, is equivalent by Theorem \ref{thm:main} to
\[
(1+q)\ch(q^{\frac{\ell(w)}{2}}C'_{w})=\ch(q^{\frac{\ell(w)+1}{2}}C'_{ws})+q\ch(q^{\frac{\ell(w)-1}{2}}C'_z).
\]
When $w$ is codominant and $ws$ is smooth, then both $ws$ and $z$ are codominant as well. Below we give an example of what happens for Hessenberg varieties.

\begin{Exa}[Geometric interpretation of the modular law for indifference graphs]
\label{exa:modularlawgeo}
  Let $\m_0$, $\m_1$, $\m_2$ be Hessenberg functions and $i\in [n]$ an integer such that $\m_0(j)=\m_1(j)=\m_2(j)$ for every $j\neq i$, $\m_0(i)=\m_1(i)-1=\m_2(i)-2$ and $\m_1(\m_1(i)+1)=\m_1(\m_1(i))$. Set $l=m_1(1)$ and let $s=(l,l+1)$ be a simple transposition. \par
  We claim that $w_{\m_1}s < w_{\m_1}< w_{\m_2} = sw_{\m_1}$, $w_{\m_0}\lessdot w_{\m_1}$ and $sw_{\m_0} < w_{\m_0}$, so we are in the hypothesis of Corollary \ref{cor:relation}. Indeed, since $\m_1(i) = l$ and $\m_1(i-1)<l$, we have that $w_{\m_1}(i) = l$, while $w_{\m_1}^{-1}(l+1) > i$. So $w_{\m_1}s < w_{\m_1}< sw_{\m_1}$. Since $\m_2(i)=l+1$ and $\m_2$ agrees with $\m_1$ everywhere else, $w_{\m_2} = sw_{\m_1}$. Finally, $w_{\m_0}\lessdot w_{\m_1}$, and  since $\m_0(i) < l$ and $\m_0(i+1)> l$, we have $sw_{\m_0}<w_{\m_0}$.

  Let $X$ be a regular semi-simple matrix, then the Hessenberg varieties are
  \begin{align*}
      \h_{\m_0}&=\{V_\bullet; XV_i\subset V_{l-1}; XV_j\subset V_{\m_1(j)}\text{ for }j\in [n]\setminus\{i\}\},\\
      \h_{\m_1}&=\{V_\bullet; XV_i\subset V_{l}; XV_j\subset V_{\m_1(j)}\text{ for }j\in [n]\setminus\{i\}\},\\
      \h_{\m_2}&=\{V_\bullet; XV_i\subset V_{l+1}; XV_j\subset V_{\m_1(j)}\text{ for }j\in [n]\setminus\{i\}\}.
  \end{align*}
Since $\m_1(l+1)=\m_1(l)$, the conditions $XV_l\subset V_{\m_1(l)}$ and $XV_{l+1}\subset V_{\m_1(l+1)}=V_{\m_1(l)}$ are redundant. In particular, there exists no condition involving $V_k$ in $\h_{\m_0}(X)$ and $\h_{\m_2}(X)$. Then the forgetful maps
\begin{align*}
    \h_{\m_0}(X)&\to \mathcal{P}_s\\
    \h_{\m_2}(X)&\to \mathcal{P}_s
\end{align*}
are $\mathbb{P}^1$-bundles over their images, which are, respectively,
\begin{align*}
    \mathcal{Z}_0&=\{\overline{V}_\bullet; X\overline{V}_i\subset \overline{V}_{l-1}, X\overline{V}_j\subset \overline{V}_{\m_1(j)}, \text{ for }j\in [n]\setminus\{i,l\}\}, \\
    \mathcal{Z}_2&=\{\overline{V}_\bullet; X\overline{V}_i\subset \overline{V}_{l+1}, X\overline{V}_j\subset \overline{V}_{\m_1(j)}, \text{ for }j\in [n]\setminus\{i,l\}\},
\end{align*}
where we write $\overline{V}_\bullet$ for a partial flag $\overline{V}_1\subset\ldots \subset \overline{V}_{l-1}\subset \overline{V}_{l+1}\subset\ldots\subset \overline{V}_n$ in $\mathcal{P}_s$. The fibers of the map $f\col\h_{\m_1}(X)\to \mathcal{Z}_2$ can be described as
\[
f^{-1}(\overline{V}_\bullet)=\{V_\bullet; V_j=\overline{V}_j\text{ for }j\in[n]\setminus\{l\}, V_{l-1}+XV_i\subset V_l\subset V_{l+1}\}
\]
So $f^{-1}(\overline{V}_\bullet)$ is isomorphic to $\mathbb{P}^1$ if $X\overline{V}_i\subset V_{l-1}$, as in this case $\overline{V}_{l-1}+X\overline{V}_i=\overline{V}_{l-1}$, or is a single point $V_\bullet$, with $V_{l}=\overline{V}_{k-1}+X\overline{V}_i$. Note that $\dim \overline{V}_{l-1}+X\overline{V}_i\leq l$, as $X\overline{V}_{i-1}\subset V_{m_1(i-1)}\subset V_{l-1}$. In fact, $\h_{\m_1}(X)$ is the blowup of $\mathcal{Z}_2$ along $\mathcal{Z}_0$.
 \[
 \begin{tikzcd}
 \h_{\m_0}(X) \ar[rrr]\ar[dd,"{\mathbb{P}^1-\text{bundle}}",swap]&&&\h_{\m_1}(X) \ar[r]\ar[dd,"{\substack{\text{Isomorphism outside }\\ \h_{\m_0}(X)}}",swap]& \h_{\m_2}(X)\ar[ldd, "{\mathbb{P}^1-\text{bundle}}"]\\ 
 \\
 \mathcal{Z}_0 \ar[rrr] &&&\mathcal{Z}_2
 \end{tikzcd}
 \]

This means that
\begin{align*}
   \ch(H^*(\h_{\m_0}(X)))&=(1+q)\ch(H^*(\mathcal{Z}_0))\\ 
   \ch(H^*(\h_{\m_1}(X)))&=\ch(H^*(Bl_{\mathcal{Z}_0}\mathcal{Z}_2))=\ch(H^*(\mathcal{Z}_2))+q\ch(H^*(\mathcal{Z}_0))\\
   \ch(H^*(\h_{\m_2}(X)))&=(1+q)\ch(H^*(\mathcal{Z}_2))
\end{align*}
and hence we get
\[
(1+q)\csf_q(\m_1)=\csf_q(\m_2)+q\csf_q(\m_1).
\]

\end{Exa}

We refer to \cite[Example 1.24]{AN_hecke} for an example where $ws$ is singular. 

A direct consequence of Example \ref{exa:modularlawgeo} is that characters of Kazhdan-Lusztig elements of codominant permutations are omega-dual to chromatic quasisymmetric functions of indifference graphs, first proved in \cite{CHSS}.
\begin{corollary}
\label{cor:hecke_csf}
If $\m\col [n]\to [n]$ is a Hessenberg function, then 
\[
\ch(q^{\frac{\ell(w_\m)}{2}}C'_{w_\m})=\omega(\csf_q(G_\m)).
\]
\end{corollary}
\begin{proof}
If $\m_0$, $\m_1$, and $\m_2$ are Hessenberg functions as in Example \ref{exa:modularlawgeo}, then applying Corollary \ref{cor:relation} to $w_{\m_1}$, we see that $w_{\m_1}s=w_{\m_2}$ and $z=w_{\m_0}$. This means that the relation in item (1) is precisely the modular law (see \cite{GPmodular} and \cite{OrellanaScott}). By \cite[Theorem 1.1]{AN}, the modular law is sufficient to characterize the values of $\ch(C'_w)$  for $w$ codominant from the values $\ch(q^{\frac{\ell(w_{\lambda})}{2}}C'_{w_\lambda})$. Since $\ch(q^{\frac{\ell(w_\lambda)}{2}}C'_{w_{\lambda}}) = \lambda!_qh_{\lambda} = \omega(G_{\m_{\lambda}})$ the result follows.
\end{proof}

\begin{Rem}
We set $H_n^{cod}$ to be the $\mathbb{C}(q^{\frac{1}{2}})$-linear subspace of $H_n$ generated by $C'_w$, for $w$ codominant. From \cite{AN}, the kernel of the linear map
 \[
 \ch\col H_{n}^{cod}\to \mathbb{C}(q^{\frac{1}{2}})\otimes \Lambda,
 \]
 is generated by the relations in Corollary \ref{cor:relation} item (1) for $w$ codominant. 
\end{Rem}

\begin{Question}
Is the kernel of the linear map $\ch \col H_n\to \mathbb{C}(q^{\frac{1}{2}})\otimes \Lambda$ generated by the relations in Equation \eqref{eq:modular}?
\end{Question}

\subsection{The geometry of $\h_w(X)$ when $w$ is smooth} In the proof of Theorem \ref{thm:mainhaimanconj} in Section \ref{sec:haimanconj}, we saw that for each smooth permutation $w\in S_n$ there exists a codominant permutation $w'$ such that the moment graphs of $\h_w(X)$ and $\h_{w'}(X)$ are the same and, in particular, they have isomorphic equivariant cohomology. We also saw that all the varieties $\h_w(X)$ associated to Coxeter elements $w$ are isomorphic. We make the following conjecture which is a strengthening of Theorem \ref{thm:mainhaimanconj}.
 
         \begin{conjecture}
         \label{conj:wcod}
            Let $X \in SL_n(\mathbb{C})$ be regular semisimple and $w\in S_n$ smooth. Then there exists a codominant permutation $w'$ such that $\h_w(X)$ and $\h_{w'}(X)$ are homeomorphic.
           \end{conjecture}
We remark that the correspoding statement for Schubert varieties is false, for instance $\Omega_{3142,F_\bullet}$ is not homeomorphic to $\Omega_{2341,F_\bullet}$ (and this is the only Schubert variety associated with a codominant permutation with the same Poincaré polynomial as of $\Omega_{3142,F_\bullet}$). On the other hand, both $3142$ and $2341$ are coxeter elements so that $\h_{3142}(X)$ is isomorphic to $\h_{2341}(X)$ if $X$ is regular semisimple.\par

\bibliographystyle{amsalpha}
\bibliography{bibli}

\end{document}